\documentclass[12pt]{elsarticle}
\usepackage{amsmath,amsthm,amssymb,bbm}

\newtheorem{theorem}{Theorem}[section]

\newtheorem{lemma}{Lemma}[section]
\newtheorem{proposition}{Proposition}[section]
\theoremstyle{definition}

\theoremstyle{remark}
\newtheorem{remark}{Remark}[section]

\newcommand{\abs}[1]{\left\lvert#1\right\rvert}
\newcommand{\norm}[1]{\left\lVert#1\right\rVert}
\newcommand{\ind}[1]{\mathbbm{1}_{#1}}
\newcommand{\set}[1]{\left\{#1\right\}}
\newcommand{\E}{\mathbf{E}}
\newcommand{\Prob}{ \mathbf{P}}
\newcommand{\eps}{\varepsilon}
\newtheorem{thm}{Theorem}[section]

\numberwithin{equation}{section}

\newcommand{\ud}{d}

\begin{document}
\title{Small ball properties and representation results}
\author[m1]{Yuliya Mishura} 
\ead{myus@univ.kiev.ua}
\author[m1]{Georgiy Shevchenko\corref{c1}}
\ead{zhora@univ.kiev.ua}
\address[m1]{National Taras Shevchenko University of Kyiv, Mechanics and Mathematics Faculty, 64~Volodymyrska, 01601 Kyiv, Ukraine}
\cortext[c1]{Corresponding author}

\begin{abstract}
 We show that small ball estimates together with H\"older continuity assumption allow to obtain new representation results in models with long memory. In order to apply these results, we establish small ball probability estimates for Gaussian  processes whose incremental variance admits two-sided estimates and  the incremental covariance preserves sign. As a result, we obtain small  ball estimates for integral transforms of  Wiener processes and of fractional Brownian motion with Volterra kernels.
\end{abstract}

\begin{keyword}
integral representation\sep
generalized Lebesgue--Stieltjes integral \sep
small ball estimate\sep
quasi-helix\sep 
fractional Brownian motion 
\MSC[2010]  60H05 \sep 60G15  \sep 60G22 
\end{keyword}

\maketitle

%\tableofcontents
\section{Introduction}

One of the most important questions for financial modeling is the question of replication, which loosely can be formulated as follows. Suppose that a continuous time financial market model is driven by a stochastic process $X= \set{X_t,t\in[0,1]}$ given on some stochastic basis $(\Omega, \mathcal{F}, \mathbb{F}= (\mathcal{F}_t)_{t \in [0,1]}, \mathsf{P})$ satisfying usual assumptions. A contingent claim, modeled by an $\mathcal{F}_1$-measurable random variable $\xi$, is hedgeable, if it admits the representation
\begin{equation}\label{equat1}
\xi=\int_0^1\psi_tdX_t,
\end{equation}
with some  $\mathbb{F}$-adapted (replicating) process $\psi$. In the case where $X$ is a Wiener process there are two main representation results. The famous It\^o representation theorem establishes \eqref{equat1} for centered square integrable random variables $\xi$. Less known is a result of Dudley \cite{dudley}, who proved that every random variable $\xi$ has representation \eqref{equat1}. There are also a lot of results for martingales or semimartingales, we will not cite them, as this is not our main concern here.

The case where $X$ is not a semimartingale is less studied. The pioneering results were established in \cite{mish-shev-valk} for fractional Brownian motion (fBm) $B^H$ with Hurst index $H>1/2$. The construction  used in \cite{mish-shev-valk} relies on the H\"older continuity and a small ball estimates for $B^H$. This fact was later used in \cite{vita-shev2,vita} to extend the results of \cite{mish-shev-valk} to a larger class of integrands. In \cite{vita-shev2}, it is also shown that in the case where $X=W+B^H$ is a sum of a Wiener process and an fBm with $H>1/2$, any random variable has representation \eqref{equat1}. It is worth to mention also the article  \cite{vita-shev}, where the existence of a continuous integrand $\psi$ is  shown in the fBm case. 

The main problem with the specific  small ball property assumed in the papers \cite{vita-shev2, vita} is that it is hard to verify. As it was mentioned in \cite{li-shao}, an upper bound in  small ball probability  gives lower estimates for metric entropy, which are usually hard to obtain. On the other hand, the assumptions of \cite{vita-shev2, vita} are not optimal for establishing representation results.

The goal of this paper is twofold.  First,  we investigate precise conditions needed to obtain the representation results and compare them to the small ball estimates. Second, we analyze carefully how to get  an upper bound for  small ball probability  for Gaussian processes with variation distance $\E|X_t-X_s|^2$ satisfying two-sided power bounds, possibly, with different powers. These two steps allow us to establish the representation results for a wide class of processes.  This class includes some  Gaussian processes $X$ having non-stationary increments, e.g.\ processes that can be  represented as  the integrals of smooth Volterra kernels w.r.t.\ a Wiener process or fBm. 
 
The paper is organized as follows. In Section \ref{sec3}, we prove representation theorems for H\"older continuous processes satisfying small ball property. In Section \ref{sec2}, we establish the small ball estimates for Gaussian  processes whose incremental variance satisfies two-sided power estimates and incremental covariance preserves sign.  In Section~\ref{sec4}, we prove representation results for the Gaussian processes considered in Section~\ref{sec2}, and  give  examples of processes, for which the representation results are in place. The examples include subfractional Brownian motion, bifractional Brownian motion, and integral transforms of Wiener process and fractional Brownian motion with Volterra kernels.

\section{Representation theorems for H\"older continuous processes satisfying small ball estimates}\label{sec3}

This section is concerned with the representation results of the form \eqref{equat1}. Here we establish general results for processes satisfying H\"older continuity and small ball assumptions. 

Consider an adapted process $X$ satisfying the following assumptions, where $C^\theta[0,1]$ denotes the class of H\"older continuous functions of order $\theta$.
\begin{itemize}
\item[$(H)$] H\"older continuity: $X\in C^\theta[0,1]$ a.s.\ for some $\theta>1/2$.
\item[$(S)$] Small ball estimate: there exist positive constants $\lambda,\mu,K_1,K_2$ such that for all $\eps>0$, $\Delta>0$, $s\in[0,1-\Delta]$
\begin{equation}\label{smallballassumption}
\mathbf{P}\left\{\sup_{s\leq t\leq s+\Delta}|X_t-X_s|\leq K_1 \varepsilon\right\}\leq\exp\left\{-K_2\varepsilon^{-\lambda}\Delta^{\mu}\right\}.
\end{equation}
\end{itemize}
\begin{remark}
In \cite{vita}, the author establishes existence of representation \eqref{equat1} for a centered Gaussian process. The assumptions of \cite{vita} are close to be a particular case of $(H)$ and $(S)$. Namely, the author assumes $(S)$ with $\lambda=1/H$, $\mu=1$. Instead of $(H)$,  the incremental variance is assumed to satisfy $\E (X_t-X_s)^2\le C|t-s|^{2H}$, which in the Gaussian case implies that $X\in C^\theta[0,1]$ for any $\theta\in(0,H)$.
\end{remark}
\begin{remark}\label{rem:exponents}
It is clear that the exponents $\lambda,\mu,\theta$ must satisfy $\theta\le \mu/\lambda$. Indeed, assume on the contrary that $\theta>\mu/\lambda$ and take arbitrary $\delta\in(\mu/\lambda,\theta)$. Then for each $n\ge 1$
$$
\mathbf{P}\left\{\sup_{0\le t\le  n^{-1}}|X_t-X_0|\le n^{-\delta}\right\}\leq\exp\left\{-K_2n^{(\lambda\delta-\mu)}\right\},
$$
whence by the Borel-Cantelli lemma, $\sup_{0\le t\le  n^{-1}}|X_t-X_0|> n^{-\delta}$ for all $n$ large enough, which contradicts $(H)$.
\end{remark}

Further we give basic facts on fractional integration; for more detail, see \cite{samko,zahle}. Consider functions  $f,g\colon[0,1]\rightarrow \mathbb{R}$, and let  $[a,b]\subset [0,1]$.
For $\alpha\in (0,1)$ define fractional derivatives
\begin{gather*}
\big(D_{a+}^{\alpha}f\big)(x)=\frac{1}{\Gamma(1-\alpha)}\bigg(\frac{f(x)}{(x-a)^\alpha}+\alpha
\int_{a}^x\frac{f(x)-f(u)}{(x-u)^{1+\alpha}}du\bigg)1_{(a,b)}(x),\\
\big(D_{b-}^{1-\alpha}g\big)(x)=\frac{e^{i\pi
\alpha}}{\Gamma(\alpha)}\bigg(\frac{g(x)}{(b-x)^{1-\alpha}}+(1-\alpha)
\int_{x}^b\frac{g(x)-g(u)}{(x-u)^{2-\alpha}}du\bigg)1_{(a,b)}(x).
\end{gather*}
Assuming that
 $D_{a+}^{\alpha}f\in L_1[a,b]$, $D_{b-}^{1-\alpha}g_{b-}\in
L_\infty[a,b]$, where $g_{b-}(x) = g(x) - g(b)$,
the generalized Lebesgue--Stieltjes
integral %$\int_a^bf(x)dg(x)$
is defined as
\begin{equation*}\int_a^bf(x)dg(x)=e^{-i\pi\alpha}\int_a^b\big(D_{a+}^{\alpha}f\big)(x)
\big(D_{b-}^{1-\alpha}g_{b-}\big)(x)dx.
\end{equation*}
It is well known that for $f\in C^\beta[a,b]$, $g\in C^\gamma[a,b]$ with $\beta+\gamma>1$, the generalized Lebesgue--Stieltjes
integral $\int_a^bf(x)dg(x)$ exists and equals the limit of Riemann sums.

In order to integrate w.r.t. $X$, fix some $\alpha \in(1-\theta,1/2)$ and introduce the following norm:
\begin{gather*}
\norm{f}_{\alpha,[a,b]} = \int_a^b \left(\frac{|{f(s)}|}{(s-a)^\alpha} + \int_a^s \frac{|{f(s)-f(z)}|}{(s-z)^{1+\alpha}}dz\right)ds.
\end{gather*}
For simplicity we will abbreviate $\norm{\cdot}_{\alpha,t} = \norm{\cdot}_{\alpha,[0,t]}$. Denote $\Lambda_\alpha:= \sup_{0\le s<t\le 1} |{D_{t-}^{1-\alpha}X_{t-}}(s)|$. In view of $(H)$, $\Lambda_\alpha<\infty$.

Then for any $t\in(0,1]$ and any $f$ such that $\norm{f}_{\alpha,t}<\infty$, the integral $\int_0^t f(s) dX_s$ is well defined as a generalized Lebesgue--Stieltjes integral, and the following estimate is clear:
\begin{gather*}
\Big|{\int_0^t f(s)dX_s}\Big|\le \Lambda_\alpha \norm{f}_{\alpha,t}.
\end{gather*}
%The next result is the key point in the proofs of the representation theorems. It was established in \cite{mish-shev-valk} for fractional Brownian motion, in \cite{vita} for Gaussian processes under rather restrictive conditions on the correlation function that can not be effectively checked for nonstationary processes with nonstationary increments, and in \cite{vita-shev} for H\"{o}lder processes admitting some small ball probabilities estimates. Now we establish this result for Gaussian processes satisfying conditions $(A1)$, $(A2)$ and $(B^+)$, and these conditions can be checked without any technical difficulties  for  a wide spectrum of processes,  as we shall see later.
\begin{lemma}\label{lemma12}
\label{lemma:aux_general}
Let an adapted  process $X$ satisfy conditions $(H)$, $(S)$. %$(A1)$, $(A2)$, $(B^+)$ with $0<2H_1-1< H_2\le H_1$.
%with
%, moreover,
%\begin{equation}\label{equat7}
%1-3H_1+H_2+H_1H_2>0.
%\end{equation}
Then there exists an adapted process $\phi$ such that $\norm{\phi}_{\alpha,t}<\infty$ for every $t<1$ and
\begin{equation}
\label{rep:aux-lemma}
\lim_{t\to 1-} \int_0^t \phi_s\ud X_s = +\infty
\end{equation}
almost surely.
\end{lemma}
%\begin{remark}
%	The assumption \eqref{rep:aux-lemma} is equivalent to $H_1 < \frac{1+H_2}{3-H_2}$. On one hand, the right hand side is greater than $H_2$ for $H_2<1$, so the assumption is compatible with $H_1\ge H_2$ (and is satisfied whenever $H_1=H_2$). On another hand, it is less than $(1+H_2)/2$, so this assumption implies negative exponents of $\eps$ in the small ball probability estimate.
%\end{remark}
\begin{remark}
	A slight modification of the argument allows to construct an integrand $\phi$ which is additionally continuous on $[0,1)$. As this is not our primary concern here, we refer to \cite{vita-shev} for an idea how the modification is carried out.
\end{remark}
\begin{proof}
Choose some $\gamma\in(1,1/\theta)$  and define $\Delta_k = K k^{-\gamma}$, $k\ge 1$, where $K = \left(\sum_{k=1}^\infty k^{-\gamma}\right)^{-1}$. Set $t_0=0$, $t_n = \sum_{k=1}^{n}\Delta_k$, $n\ge 1$. Then $t_n\uparrow 1$, $n\to\infty$.

Further define the sequence of continuously differentiable functions $g_n(x) = \sqrt{x^2 + 4^{-n}} - 2^{-n}$, $n\ge 1$. Obviously, $|x|\ge g_n(x)\ge (|x|-2^{-n})\vee 0$. Finally,
fix some $\beta > \mu/\lambda$,
introduce a sequence of stopping times
\begin{equation*}
\tau_n = \min\left\{t\geq t_{n-1} : |X(t) - X({t_{n-1}})| \geq n^{-\beta}\right\} \wedge t_n,
\end{equation*}
and  set
$$
\phi_s = \sum_{k=1}^{\infty} k^{\beta-1}g'_{k}(X_s - X_{t_{k-1}})\textbf{1}_{[t_{k-1},\tau_k)}(s).
$$

We will check that $\phi$ is as required. The finiteness of the norm $\norm{\phi}_{\alpha,t}$ is shown exactly as in \cite{mish-shev-valk} and therefore will be omitted. Thanks to the change of variable integration formula for the generalized Lebesgue--Stieltjes integral (see e.g.\ \cite{Mish}), for $t\in[t_{k-1},t_{k})$
\begin{gather*}
\int_{t_{k-1}}^t \phi_s dX_s = k^{\beta-1} g_k(X_{t\wedge \tau_k}-X_{t_{k-1}}).
\end{gather*}
Then for any $n\ge 1$ and  $t\in[t_{n-1},t_{n})$
\begin{gather*}
\int_0^t\phi_s dX_s = \sum_{k=1}^{n-1} \int_{t_{k-1}}^{t_k} \phi_s dX_s +  \int_{t_{n-1}}^t \phi_s dX_s \\
= \sum_{k=1}^{n-1} k^{\beta-1} g_k(X_{\tau_k}-X_{t_{k-1}}) +  n^{\beta-1}g_n(X_{t\wedge \tau_n}-X_{t_{n-1}})\\
\ge \sum_{k=1}^{n-1} k^{\beta-1}\big(|X_{\tau_k}-X_{t_{k-1}}|- 2^{-k}\big) \ge \sum_{k=1}^{n-1} k^{\beta-1}|X_{\tau_k}-X_{t_{k-1}}|- \sum_{k=1}^\infty k^{\beta-1} 2^{-k}.
\end{gather*}
In order to prove the claim, we need to show that the series $\sum_{k=1}^{n-1} k^{\beta-1}|X_{\tau_k}-X_{t_{k-1}}|$ diverges. To this end it suffices to show that, almost surely, $\tau_k <t_k$ for all $k$ large enough so that $k^{\beta-1}|X_{ \tau_k}-X_{t_{k-1}}| = k^{-1}$ eventually.

Clearly, $\tau_k= t_k$ implies
$\sup_{t\in[t_{k-1},t_k]}|X_{t}-X_{t_{k-1}}|\le k^{-\beta}$, so by  $(S)$,
$$
\mathbf{P}(\tau_k=t_k)\le K_1\exp\left\{ - K_2 k^{\beta \lambda -\gamma \mu }\right\}.
$$
If  the exponent near $k$ is positive, then we are done. The positivity  is easily seen to be  equivalent to $\gamma < \beta \lambda/\mu$.
On the other hand, we must have $\gamma>1$. Due to the choice of $\beta$, we can choose $\gamma$ satisfying  both requirements, thus finishing the proof.
\end{proof}

We are ready to state the main result of this section. While its proof heavily borrows from \cite{mish-shev-valk,vita-shev,vita-shev2}, we decided nevertheless to give it for two reasons. Firstly, we aimed to keep the article self-consistent. Secondly, we desired to stress all the key points of the proof in order to make sure that the assumptions are optimal.
\begin{theorem}\label{thm:general}
Let an adapted  process $X$ satisfy conditions $(H)$, $(S)$, and a random variable $\xi$ be such that $\xi=Z_1$ for some adapted process $Z$ such that $Z\in C^\rho[0,1]$ with  $\rho >  \mu/(\lambda\theta)-1$ a.s. Then there exists an adapted process $\psi$  such that $\norm{\psi}_{\alpha,1}<\infty$ for some $\alpha\in (1-\theta,1/2)$ and
\begin{equation}\label{eq:represent}
\int_0^1 \psi_s\ud X_s = \xi
\end{equation}
almost surely.
\end{theorem}
\begin{remark}
One might hope to get this representation result for H\"older continuous process of any order provided that $\mu/(\lambda\theta)\le 1$ (so that the restriction on $\rho$ is void). This, however, is possible only if $\theta = \mu/\lambda$, as it was explained in Remark~\ref{rem:exponents}.
\end{remark}
\begin{proof}
Let  $\set{t_n,n\ge 1}\in(0,1)$ be some sequence of points such that $t_n\uparrow 1$, $n\to\infty$. We will  construct an adapted process $\psi$ such that
\begin{itemize}
\item[$(\Phi1)$] For all $n$ large enough $\int_0^{t_n}\psi_s dX_s = Z_{t_{n-1}}$.
\item[$(\Phi2)$]  $\norm{\psi}_{\alpha,[t_n,1]}\to 0$, $n\to\infty$.
\end{itemize}
Since $Z_{t_n}\to Z_1$, $n\to\infty$, by continuity, these properties imply \eqref{eq:represent}.
%To facilitate reading, we divide the rest of the proof into several steps.
%\underline{Step 1. Construction}.

 Denote for $n\ge 1$ \ $\xi_n = Z_{t_{n}}$, $\Delta_n = t_{n+1}-t_n$, $\delta_n = |\xi_{n}-\xi_{n-1}|$.

We construct the process $\psi$ inductively on $[t_n,t_{n+1}]$. To this end, we take some positive sequences $\set{\sigma_n,n\ge 1}$ and $\set{\nu_n,n\ge 1}$ such that $\sigma_n\to\infty$, $n\to\infty$.

We start the construction setting $\psi_t = 0$ for $t\in[0,t_1]$.
Further, assume that $\psi$ is constructed on $[0,t_n)$ and denote $V_t = \int_0^t \psi_s dX_s$. The construction will depend on whether some event $A_n\in \mathcal F_{t_n}$, which will be specified later, or its complement $B_n=\Omega\setminus A_n$ holds.

Case 1:  $\omega \in A_n$. Thanks to Lemma~\ref{lemma:aux_general}, there exists a process $\set{\phi_t,t\in [t_n,t_{n+1}]}$ such that $\int_{t_n}^t \phi_s dX_s \to +\infty$, $t\to t_{n+1}-$.
Define $v_n = V_{t_n}-\xi_{n}$, $$\tau_n = \inf\set{t\ge t_n: \int_{t_n}^t \phi_s dX_s \ge |v_n|}$$
and set
$$
\psi_s = \phi_s \operatorname{sign}v_n\,\ind{[t_n,\tau_n]}(t),\  t\in[t_n,t_{n+1}].
$$
It is clear that $\int_{t_n}^{t_{n+1}} \phi_s dX_s = v_n$, hence, $V_{t_{n+1}} = V_{t_n} + v_n = \xi_n$.

Case 2: $\omega\in B_n$. Similarly to the proof of Lemma~\ref{lemma:aux_general}, define $g_n(x) = \sqrt{x^2 + \nu_n^2} - \nu_n$ so that $g_n\in C^\infty(\mathbb{R})$, $|x|\ge g_n(x)\ge (|x|-\nu_{n})\vee 0$. Introduce the stopping time
$$
\tau_n = \inf\set{t\ge t_n: \sigma_n g_n(X_{t}-X_{t_n})\ge \delta_n}\wedge t_{n+1}
$$
and set
$$
\psi_s = \sigma_n g'_n(X_t-X_{t_n}) \operatorname{sign}(\xi_n-\xi_{n-1})\ind{[t_n,\tau_n]}(t), t\in[t_n,t_{n+1}].
$$

By the change of variable formula for the generalized Lebesgue--Stieltjes integral,
$$V_{t_{n+1}}-V_{t_n} = \int_{t_n}^{t_{n+1}} \psi_s dX_s = \sigma_n g_n(X_{t_{n+1}\wedge \tau_n} - X_{t_n}).
$$
Therefore, $V_{t_{n+1}}-V_{t_n} = \xi_n -\xi_{n-1}$ provided that $\tau\le \tau_n$. In turn, thanks to the properties for $g_n$, the latter holds
if $\sigma_n\sup_{t\in[t_{n-1},t_n]}|X_{t}-X_{t_{n-1}}|\le \delta_n + \sigma_n\nu_n$.
In view of this, define
$$
A_{n+1} = B_n\cap \set{\sup_{t\in[t_{n-1},t_n]}|X_{t}-X_{t_{n-1}}|\le \delta_n\sigma_n^{-1} + \nu_n}, n\ge 1,
$$
and $A_1 = \Omega$.

%\underline{Step 2. }
Now we identify conditions under which this construction works. First note that for $(\Phi1)$ it is suffices to ensure that
\begin{equation}\label{eq:borelcantelli}
\Prob\left(\limsup_{n\to\infty} A_n\right)=0.
\end{equation}
Indeed, let $N = N(\omega) = \max\set{n\ge 1: \omega\in A_n}$. Then by construction, $V_{t_{N+1}}=\xi_N$. Moreover,  since $\omega\in B_n$ for all $n\ge N+1$, we have $V_{t_{n+1}}-V_{t_n}=\xi_n-\xi_{n-1}$ for all $n\ge N+1$, whence $(\Phi1)$ follows.

Now turn to $(\Phi2)$. Write for $n\ge N$
$$
\norm{\psi}_{\alpha, [t_n,1]} = I_1 + I_2,
$$
where
$$
I_1 =  \int_{t_n}^1\frac{\abs{\psi_t}}{(t-t_n)^\alpha} \ud s,\quad I_2  = \int_{t_n}^1\int_{t_n}^{t}\frac{\abs{\psi_t-\psi_s}}{(t-s)^{\alpha+1}}\ud s\,\ud t.
$$
%\int_{\tau_n}^1 \left(\frac{\abs{\psi_t}}{(t-\tau_n)^\alpha} + \int_{\tau_n}^{t}\frac{\abs{\psi_t-\psi_s}}{(t-s)^{\alpha+1}}\right)\ud s
Estimate, taking into account that $\abs{g'_k}\le 1$,
\begin{align*}
I_1 & =  \sum_{k=n}^{\infty} \int_{t_k}^{t_{k+1}}\frac{\abs{\psi_t}}{(t-t_n)^{\alpha}}\ud t \le \sum_{k=n}^{\infty} \int_{t_k}^{t_{k+1}} \frac{\sigma_k}{(t - t_{k})^{\alpha}} \le C\sum_{k=n}^{\infty} \sigma_k\Delta_k^{1-\alpha}.
\end{align*}
Further, denoting $\psi(t,s) = \abs{\psi_t-\psi_s}(t-s)^{-\alpha-1}$ and taking into account that $\psi_t = 0$ for $t\in(\tau_k,t_{k+1}]$, write
\begin{align*}
I_2 & = % \int_{\tau_n}^1\int_{\tau_n}^{t}\frac{\abs{\psi_t-\psi_s}}{(t-s)^{\alpha+1}}\ud s\,\ud t = \\
\sum_{k=n}^\infty \int_{t_k}^{\tau_{k}}  \int_{t_n}^{t} \psi(t,s)\ud s \,\ud t
+ \sum_{k=n}^\infty \int_{t_{k}}^{t_{k+1}}  \int_{t_n}^{t} \psi(t,s)\ud s \,\ud t\\
&  = \sum_{k=n}^\infty \int_{t_k}^{\tau_{k}}  \int_{t_n}^{t_k} \psi(t,s)\ud s \,\ud t
+ \sum_{k=n}^\infty \int_{\tau_{k}}^{t_{k+1}}  \int_{t_n}^{\tau_{k}} \psi(t,s)\ud s \,\ud t\\
& + \sum_{k=n}^\infty \int_{t_k}^{\tau_{k}}  \int_{t_k}^{t} \psi(t,s)\ud s \,\ud t
=: J_1 + J_2 + J_3.
\end{align*}
Estimate the terms separately:
\begin{align*}
&J_1\le 2\sum_{k=n}^{\infty} \sigma_k\int_{t_k}^{\tau_{k}}  \int_{t_n}^{t_k}\frac{\ud s\, \ud t}{(t-s)^{\alpha + 1}} \le C\sum_{k=n}^{\infty} \sigma_k\int_{t_k}^{t_{k+1}}\frac{\ud t}{(t-t_n)^{\alpha}}
 \le C \sum_{k=n}^{\infty}\sigma_k {\Delta}_k^{1-\alpha}.
\end{align*}
Similarly, $J_2  \le C \sum_{k=n}^{\infty}\sigma_k {\Delta}_k^{1-\alpha}$.
To estimate $J_3$, denote $r_k = \nu_k^{1/\theta}$ and
decompose
$$
J_3  = \sum_{k=n}^\infty\int_{t_k}^{\tau_{k}}\left(\int_{t_k}^{t-r_k}+\int_{t-r_k}^{t}\right)\psi(t,s)\ud s \,\ud t=:J_{31}+J_{32}.
$$
Then
\begin{align*}
J_{31} &\le 2\sum_{k=n}^\infty \sigma_k\int_{t_k}^{\tau_{k+1}}\int_{t_k}^{t-r_k}(t-s)^{-\alpha-1}\ud s\, \ud t\\&\le
 C\sum_{k=n}^\infty \sigma_k\int_{t_k}^{\tau_{k}}r_k^{-\alpha}\ud t = C\sum_{k=n}^{\infty}  \sigma_k \Delta^{\vphantom{H}}_k \nu_k^{-\alpha/\theta}.
\end{align*}
To estimate $J_{32}$, note that $\abs{g_k''(x)}\le \nu_k^{-1}$. Therefore,
\begin{align*}
&J_{32}\le \sum_{k=n}^{\infty}
\sigma_k \nu_k^{-1}\int_{t_k}^{\tau_{k}}  \int_{t-r_k}^{t}\frac{\abs{X_t-X_s}}{(t-s)^{\alpha+1}}\ud s \,\ud t\\
& \le C\sum_{k=n}^{\infty}\sigma_k \nu_k^{-1}\int_{t_k}^{\tau_{k}}  \int_{t-r_k}^{t}(t-s)^{\theta-\alpha-1}\ud s \,\ud t
\\&\le C \sum_{k=n}^{\infty} \sigma_k \nu_k^{-1} \Delta^{\vphantom{H}}_k r_k^{\theta-\alpha} = C\sum_{k=n}^{\infty}  \sigma_k^{\vphantom{H}} \Delta^{\vphantom{H}}_k \nu_k^{-\alpha/\theta}.
\end{align*}
Note that the estimation of the summand should be modified for $r_k>\Delta_k$, i.e.\ $\nu_k>\Delta_k^{\theta}$. In this case the summand is bounded by $\sigma_k \nu_k^{-1}\Delta_k^{\theta-\alpha + 1}<\sigma_k^{\vphantom{H}} \Delta_k^{\vphantom{H}} \nu_k^{-\alpha/\theta}$, which leads to the same estimate.

Summing up, we get that $\norm{\psi}_{\alpha, [t_n,1]}\to 0$, $n\to\infty$, iff
\begin{equation}\label{eq:series}
\sum_{n=1}^{\infty} \sigma_n^{\vphantom{H}}\Delta_n^{1-\alpha}<\infty\ \text{ and }\ \sum_{n=1}^{\infty}  \sigma_n^{\vphantom{H}} \Delta^{\vphantom{H}}_n \nu_n^{-\alpha/\theta}<\infty.
\end{equation}

Let us discuss the choice of parameters. First we need to ensure \eqref{eq:borelcantelli}. Suppose that $\kappa\in (0,\rho)$. Then the H\"older assumption $Z\in C^\rho$ implies that $\delta_n = o(\Delta_n^\kappa)$, $n\to\infty$ a.s. Setting $\nu_n = \Delta_n^\kappa\sigma_n^{-1}$,  we get that
$$
\limsup_{n\to\infty} A_n\subset \limsup_{n\to\infty} C_n,
$$
where
$$
C_{n} = \set{\sup_{t\in[t_{n-1},t_n]}|X_{t}-X_{t_{n-1}}|\le 2\Delta_n^\kappa\sigma_n^{-1}},\ n\ge 1.
$$
Then, but virtue of the Borel--Cantelli lemma, it suffices to ensure that
$\sum_{n=1}^{\infty} \Prob(C_n)<\infty$. In view of the small ball estimate $(S)$,
$$
\Prob (C_n) \le K_1\exp\set{-2^{-\mu}K_2 \Delta_n^{\mu-\kappa\lambda}\sigma_n^{\lambda}}.
$$
Taking $\sigma_n = n^{\epsilon}\Delta_n^{\kappa-\mu/\lambda}$ with $\epsilon>0$, we get $\sum_{n=1}^{\infty} \Prob(C_n)<\infty$, as required.

Now turn to \eqref{eq:series}. With the above choice of $\sigma_n$ and $\nu_n$,  they  transform to
$$
\sum_{n=1}^{\infty} n^{\epsilon}\Delta_n^{1+ \kappa - \mu/\lambda -\alpha}<\infty
$$
and
$$
\sum_{n=1}^{\infty}  n^{\tau} \Delta^{1+\kappa - \mu(1+\alpha/\theta)/\lambda}_n<\infty,
$$
where $\tau = \epsilon(1+\alpha/\theta)$. Taking $\Delta_n=2^{-n}$, it is enough to make the both exponents near $\Delta_n$ positive. Since $\mu/(\lambda\theta)\ge 1$, the second exponent is smaller, so we end up with the requirement that
\begin{equation*} %\label{eq:exponentsrestr}
1+\kappa - \mu(1+\alpha/\theta)/\lambda>0.
\end{equation*}
The other restrictions we have are $\kappa<\rho$ and $\alpha>1-\theta$.  So the choice of $\kappa$ and $\alpha$ is possible iff
$$
1+\rho - \mu(1+(1-\theta)/\theta)/\lambda>0,
$$
which is easily seen to be equivalent to $\rho> \mu/(\lambda\theta)-1$. The proof is now complete.
\end{proof}
\begin{remark}\label{rem:equivalence}
With our approach, the assumption that $Z$ is H\"older continuous is unavoidable. Indeed, in order for the argument to work, one must have $\delta_n \sigma_n^{-1} = O(\Delta_n^\theta)$, $n\to\infty$, as in the opposite case we would have a contradiction with $(H)$. On the other hand, the series $\sum_{n=1}^{\infty} \sigma_n \Delta_n^{1-\alpha}$ must converge. Consequently, $\sum_{n=1}^{\infty} \delta_n \Delta_n^{1-\theta-\alpha}<\infty$, whence $\delta_n = o(\Delta_n^{\theta+\alpha-1})$, $n\to\infty$, as claimed.
\end{remark}
\begin{remark}
Assumption $(S)$ is not optimal for establishing Theorem~\ref{thm:general}. It is easy to see that the conclusion is of ``probability zero'' spirit, in particular, it does not change with switching to an equivalent measure. Assumption $(S)$ is more delicate and  in general will not hold for an equivalent measure. It is possible to formulate a relevant ``almost sure'' assumption, for example: 
\begin{itemize}
\item[$(S')$] There exists a number $a>0$ such that for any sequence of points $\set{t_n,n\ge1}\in(0,1)$ such that $t_{n}\uparrow 1$, $n\to\infty$, $t_{n+1} - t_n\sim K n^{-\gamma}$ with some $K>0$, $\gamma>1$, it holds 
$$
\lim_{n\to\infty} n^{\gamma a}\sup_{t\in[t_{n-1},t_n]}|X_{t}-X_{t_{n-1}}|=+\infty.
$$
\end{itemize}
It is easy to check that if one assumes $(S')$ instead of $(S)$, then Theorem~\ref{thm:general} holds with $\rho_0 = a/\theta-1$. Also, similarly to the proof of \eqref{eq:borelcantelli}, $(S')$ implies $(S)$ with any $a>\mu/\lambda$.

However,  $(S')$ is not easy to check. Alternatively, one can assume that the distribution of $X$ is equivalent to that of a process satisfying $(S)$, which seems more natural. However, one needs to ensure that the adaptedness is preserved with the change of measure, e.g.\ by using some version of the Girsanov theorem.
\end{remark}

 \section{Small ball probability estimates and representation results for Gaussian processes}\label{sec2}

 In this section we first establish the small ball property for Gaussian processes satisfying two-sided estimates on the incremental variance  and preserving the sign of  incremental covariance. We remark that similar assumptions on the incremental variance were imposed in \cite{azmoo-vita}, however, the assumptions on the covariance differ significantly, so the findings are different. Using the small ball estimates, we derive representation results for such Gaussian processes. Finally we give the examples of the processes satisfying these conditions, including integral transforms with Volterra kernels of a Wiener process and of an fBm. 

 \subsection{Small ball property for Gaussian properties with  variance distance satisfying two-sided estimates}
 Let $X=\{X_t, t\in [0,1]\}$ be a centered Gaussian process on a finite interval $[0,1]$, whose variance distance $\E|X_t-X_s|^2$  satisfies the following two-sided power bounds:
 \begin{itemize}
 \item[$(A_1)$] There exist $H_1\in(0,1]$ and  $C_1>0$ such that for any $s,t\in[0,1]$
 $$ \E(X_t-X_s)^2\ge C_1|t-s|^{2H_1}. $$
 \item[$(A2)$] There exist $H_2\in(0,1]$ and  $C_2>0$ such that for any $s,t\in[0,1]$
 $$\E(X_t-X_s)^2\le C_2|t-s|^{2H_2}. $$
 \end{itemize}
Clearly, $H_1\ge H_2$.  In the particular case where $H_1=H_2$, such process is called a quasi-helix, see \cite{kahane,kahanebook}.

Furthermore, assume that the increments of $X$ are either positively or negatively correlated.  More precisely, we assume one of the following conditions:
 \begin{itemize}
 \item[$(B^{\pm})$] For any  $s_1,t_1,s_2,t_2\in[0,1]$, $ s_1\leq t_1\leq s_2\leq t_2$  $$\pm \E(X_{t_1}-X_{s_1})(X_{t_2}-X_{s_2})\geq 0.$$
  \end{itemize}
\begin{remark}\label{remplusminus}
It is worth to mention that $(A1)$ and $(B^+)$ imply that $H_1\ge 1/2$. Indeed, %take for simplicity  $\mathbb{T} = [0,1]$ and
write
for any $n\ge 1$
\begin{gather*}
C_2\ge \E(X_1-X_0)^2 = \sum_{k=1}^n \E(X_{k/n}-X_{(k-1)/n})^2\\ + \sum_{i\neq j} \E[(X_{i/n}-X_{(i-1)/n})(X_{j/n}-X_{(j-1)/n})]
\ge C_1 n^{1-2H_1},
\end{gather*}
whence  the claim follows by letting $n\to \infty$. Similarly, $(A2)$ and $(B^-)$ imply that $H_2\le 1/2$.
\end{remark}

Further, introduce some notations for different constants. More precisely, denote \begin{gather*}
 	C_0=\frac{64}{C_1},\
 C_3=\frac{C_0^{H_1/2}}{4^{H_1}},\ C_4=\left(16 C^2_2 C_0^{({H_2+1})/{2H_1}}\right)^{-1},\
 C_5 = \left(32 C^2_2 C_0^{({4H_2+1})/{2H_1}}\right)^{-1}.
% \\
% C_6= {C_1}(2C_2C_3^{\frac{2H_2-1}{2H_1}})^{-1}, C_7= {C_4^{\frac{2H_2-1}{2H_1(H_1-H_2)}}}{C_6^{\frac{1}{2(H_2-H_1)}}},\ C_8=\left(32C_2^2C_3^{\frac{4H_2+1}{2H_1}}\right)^{-1}.
%
\end{gather*}

The following theorem establishes an upper bound for  small deviations of the process $X$.
 \begin{theorem}\label{lemma1} Let $X=\{X_t, t\in [0,1]\}$ be a Gaussian process satisfying  $(A1)$ and $(A2)$.
 \begin{itemize}
  \item[$(1)$] If  $(B^+)$ holds, then for any
  $\varepsilon\in(0,C_3]$
 \begin{equation}\label{equa}
 \mathbf{P}\left\{\sup_{0\leq s\leq t\leq 1}|X_t-X_s|\leq \varepsilon\right\}\leq\exp\left\{-C_4\varepsilon^{4-(2H_2+2)/H_1}\right\}. \end{equation}
 \item[$(2)$] If $(B^-)$ holds, then for any $\varepsilon\in(0,C_3]$
 $$\mathbf{P}\left\{\sup_{0\leq s\leq t\leq 1}|X_t-X_s|\leq \varepsilon\right\}\leq\exp\left\{-C_5\varepsilon^{4-({4H_2+1})/{H_1}}\right\}.$$
\end{itemize}
 \end{theorem}
 \begin{remark}\label{remark 2.2}
 	The small ball estimates of Theorem~\ref{lemma1} are useful only whenever the exponents of $\varepsilon$ are negative. It is easy to see that in case (1) this happens if $H_2> 2H_1-1$, in case (2), if $H_2> H_1-1/4$. Recall also that in case (1), $H_1\ge 1/2$, in case (2), $H_2\le 1/2$; in both cases $H_2\le H_1$.
 \end{remark}

 \begin{proof}
 We use Theorem 4.4. from \cite{li-shao}. According to this result, for any
 any $a\in(0,1/2]$ and any $\varepsilon>0$ the inequality
 \begin{equation}\label{equat4} \mathbf{P}\left\{\sup_{0\leq s \leq t\leq 1}|X_t-X_s|\leq \varepsilon\right\}\leq\exp\left\{-\frac{\varepsilon^{4}}{16a^2\sum_{2\leq i,j\leq a^{-1}}(\E(\xi_i\xi_j))^2}\right\}
 \end{equation}
holds provided that
\begin{equation}\label{32eps2}
a\sum_{2\leq i\leq a^{-1}}\E\xi_i^2\geq 32\varepsilon^2,
\end{equation}
where $\xi_i=X_{ia}-X_{(i-1)a}$. Thanks to  $(A1)$ and $(A2)$,
 \begin{equation}\label{equat5}
 C_1 a^{2H_1}\leq \E\xi^2_i\leq C_2 a^{2H_2}.
 \end{equation}
Therefore,  $a\sum_{2\leq i\leq a^{-1}}\E\xi_i^2\geq C_1 a\left(\left[a^{-1}\right]-1\right) a^{2H_1}\geq C_1 \left(1-2a\right)a^{2H_1}$.
If $a\le\frac14$, inequality \eqref{32eps2}  holds whenever $a^{2H_1}\geq 64{C_1}^{-1}\varepsilon^2=C_0\varepsilon^2.$  Thus, we get the estimate \eqref{equat4} by setting $a = C_0^{1/2H_1}\varepsilon^{1/H_1}$ (the assumption $\varepsilon\le C_3$ entails that $a\le 1/4$).

(1) If $(B^+)$ holds, then
 \begin{equation}\begin{gathered}\label{equat2}
 \sum_{2\leq i,j\leq a^{-1}}(\E(\xi_i\xi_j))^2\leq \max_{2\le i,j\le a^{-1}}|\E(\xi_i\xi_j)| \sum_{2\leq i,j\leq a^{-1}} \E \xi_i \xi_j\\ \leq  C_2 \max_{2\le i\le a^{-1}}\E \xi_i^2 \ \E \bigg(\sum_{2\le i\le  a^{-1}}\xi_i\bigg)^2 = C_2 a^{2H_2} \E (X_{a[a^{-1}]}-X_{a})^2\le C^2_2 a^{2H_2}.
 \end{gathered}\end{equation}
 Substituting \eqref{equat2} into \eqref{equat4}, we arrive at the desired estimate.

(2) If  $(B^-)$ holds, then  \begin{equation}\begin{gathered}\label{equat6}\sum_{2\leq i,j\leq  a^{-1}}\big(\E(\xi_i\xi_j)\big)^2
\leq \max_{2\leq i,j\leq  a^{-1}}|\E(\xi_i\xi_j)| \left(\sum_{2\leq i\leq  a^{-1}} \E \xi_i^2 -2\sum_{2\leq i<j\leq  a^{-1}} \E \xi_i \xi_j\right)\\
\le C_2 a^{2H_2} \left(2\sum_{2\leq i\leq  a^{-1}} \E \xi_i^2-\sum_{2\leq i,j\leq  a^{-1}} \E \xi_i \xi_j\right)\\
=C_2 a^{2H_2}\left(2\sum_{2\leq i\leq  a^{-1}} \E \xi_i^2-\E(X_{a[a^{-1}]} -X_{a})^2\right)
\le 2C_2^2 a^{4H_2-1}.
 \end{gathered}\end{equation}
Plugging \eqref{equat6} into \eqref{equat4} and recalling that $a = C_0^{1/2H_1}\varepsilon^{1/H_1}$, we get the desired estimate.
 \end{proof}

As a corollary, we establish a small ball property on any interval.
\begin{proposition}\label{smallballprop}
Let $X=\{X_t, t\in [t_0,t_0+\Delta]\}$ be a Gaussian process satisfying  $(A1)$ and $(A2)$.
\begin{itemize}
	\item[$(1)$] If  $(B^+)$ holds, then for any
	$\varepsilon\in(0,C_3\Delta^{H_1}]$
	$$\mathbf{P}\left\{\sup_{t_0\leq s\leq t\leq t_0+\Delta}|X_t-X_s|\leq \varepsilon\right\}\leq\exp\left\{-C_4\varepsilon^{4-(2H_2+2)/H_1}\Delta^{2-2H_2}\right\}. $$
	\item[$(2)$] If $(B^-)$ holds, then for any $\varepsilon\in(0,C_3\Delta^{H_1}]$
	$$\mathbf{P}\left\{\sup_{t_0\leq s\leq t\leq t_0+\Delta}|X_t-X_s|\leq \varepsilon\right\}\leq\exp\left\{-C_5\varepsilon^{4-({4H_2+1})/{H_1}}\Delta\right\}.$$
\end{itemize}
\end{proposition}
\begin{proof}
	Define $X'_t = \Delta^{-H_1} (X_{t_0+\Delta t}-X_{t_0})$, $t\in[0,1]$. Then $X'$  satisfies $(A1)$ on $[0,1]$. It also satisfies (A2), but with a different constant, namely, $C_2' = C_2 \Delta^{2(H_2-H_1)}$. Setting $\eps' = \eps \Delta^{-H_1}$ and applying Theorem~\ref{lemma1}, we arrive at the required statement.
\end{proof}

\section{Representation results for Gaussian processes}\label{sec4}
Now we apply the obtained representation results to processes from Section~\ref{equat1}. We could omit the following auxiliary result and derive the required results directly from Theorem~\ref{thm:general}. However, we give it not only for the sake of completeness, but also to identify relation between assumptions $(H)$, $(S)$ and those from Section~\ref{sec2}.
\begin{lemma}
Assume that an adapted Gaussian process $X$ satisfies conditions $(A1)$, $(A2)$, $(B^+)$ with $0<2H_1-1< H_2\le H_1$.
%, moreover,
%\begin{equation}\label{equat7}
%1-3H_1+H_2+H_1H_2>0.
%\end{equation}
Then there exists an adapted process $\phi$ such that $\norm{\phi}_{\alpha,t}<\infty$ for every $t<1$ and
\begin{equation*}
\lim_{t\to 1-} \int_0^t \phi_s\ud X_s = +\infty
\end{equation*}
almost surely.
\end{lemma}
\begin{remark}
It is not possible to state similar results for processes satisfying $(A1)$, $(A2)$ and $(B^{-})$, since $(H)$ with $\theta>1/2$ requires that $H_2>1/2$.
\end{remark}
\begin{proof}
From $(A2)$ it follows that $X$ is H\"older continuous of any order less than $H_2$, so $(H)$ holds.  Proposition~\ref{smallballprop} (1) yields
\begin{equation}\label{eq:smallballrem}
\mathbf{P}\left\{\sup_{s\leq t\leq s+\Delta}|X_t-X_s|\leq  \varepsilon\right\}\leq\exp\left\{-C_4\varepsilon^{-\lambda}\Delta^{\mu}\right\}.
\end{equation}
with $\lambda=(2H_2+2)/H_1 - 4  = 2(H_2+1-2H_1)>0$, $\mu = 2-2H_2>0$ for all $\eps\in (0,C_3 \Delta^{H_1}]$.  Setting $K_1 = \exp\set{C_4 C_3^{-\lambda}}$, $K_2 = C_4$,   we obtain \eqref{smallballassumption}: for $\eps\in (0,C_3 \Delta^{\mu/\lambda}]$ use \eqref{eq:smallballrem} and the observation that $H_1 \le \mu/\lambda$, for $\eps>0$, the left-hand side exceeds $1$. Thus, the statement follows from  Lemma~\ref{lemma:aux_general}.
\end{proof}

The following result is a consequence of Theorem~\ref{equat1}.
\begin{thm}\label{cor:main}
Assume that an adapted Gaussian process $X$ satisfies conditions $(A1)$, $(A2)$, $(B^+)$ with $0<2H_1-1< H_2\le H_1$. Let also a random variable $\xi$ be such that $\xi=Z_1$ for some adapted process $Z\in C^\rho[0,1]$ with  $\rho > \rho_0$, where
\begin{equation}\label{rho0}
\rho_0 = \frac{(1+H_2)(H_1-H_2)}{H_2+1-2H_1}.
\end{equation}
Then there exists an adapted process $\psi$  such that $\norm{\psi}_{\alpha,1}<\infty$ for some $\alpha\in (1-H_2,1/2)$ and
\begin{equation*}
\int_0^1 \psi_s\ud X_s = \xi
\end{equation*}
almost surely.
\end{thm}
\begin{remark}
Consider the case where $H_1=H_2$. The processes $X$ satisfying $(A1)$ and $(A2)$ are called quasi-helices, see \cite{kahane,kahanebook}.  For such processes  $\rho_0=0$, so all values of $\rho$ are possible, which agrees with the results of \cite{mish-shev-valk,vita}.
\end{remark}
\begin{remark}
It is natural to study the representation question in the case where $\xi=Z_1$, and the process $Z$ has the same regularity as $X$. In the general case this translates to the requirement that $\theta(1-\theta)>\mu/\lambda$. In the particular case of Corollary \ref{cor:main} this translates to the inequality
$$
H_1 < \frac{2H_2(H_2+1)}{1+3H_2}.
$$
A simpler sufficient condition for this is that $H_1\le 4H_2/5 - 1/5$.
\end{remark}

\subsection{Examples}
Here we present some examples of processes which satisfy the assumptions $(A1)$, $(A2)$, $(B^+)$ so that the representation result of Theorem~\ref{cor:main} is true. We remark that it is enough to require the properties to hold on a subinterval $[t_0,1]$. 
\subsubsection{Subfractional Brownian motion}
Recall the definition of subfractional Brownian motion: this is a centered Gaussian process $G^H = \set{G^H_t,t\ge 0}$ with the covariance function 
$$
\E G^H_t G^H_u = t^{2H} + s^{2H} -\frac{1}{2}\left((t+s)^{2H} + \abs{t-s}^{2H}\right),
$$
where $H\in(0,1)$ is the self-similarity parameter of $G^H$, a counterpart of the Hurst parameter of fractional Brownian motion. It is easy to check that the increments of $G^H$ are not stationary. 

It is proved in  \cite{bojdecki} that $G^H$ satisfies properties $(A1)$, $(A2)$ with $H_1=H_2=H$, so $G^H$ is a quasi-helix. It also satisfies  $(B^+)$ for $H\in (1/2,1)$, and $(B^-)$ for $H\in(0,1/2)$. Consequently, the conclusion of Theorem~\ref{cor:main} holds for $G^H$ with any $\rho>0$.

\subsubsection{Bifractional Brownian motion}
The bifractional Brownian motion is a centered Gaussian process $B^{H,K} = \set{B^{H,K}_t,t\ge 0}$ with the covariance function
$$
R(t,s) = \E B^{H,K}_t B^{H,K}_u = \frac{1}{2^K}\left( \left(t^{2H} + s^{2H}\right)^K - \abs{t-s}^{2HK}\right),
$$
where $H\in(0,1)$, $K\in(0,1]$. This is an $HK$-self-similar process with non-stationary increments, which is also a quasi-helix, that is, it satisfies $(A1)$ and $(A2)$ with $H_1=H_2=HK$ (see \cite{houdre}).

Concerning $(B^+)$, assume that $HK>1/2$ and write
\begin{gather*}
\frac{\partial R(t,s)}{\partial t\partial s}  = \frac{2HK}{2^K}\left(2H(K-1)t^{2H-1}s^{2H-1}\left(t^{2H}+s^{2H}\right)^{K-1} + (2HK-1)\abs{t-s}^{2HK-2} \right) \\
= C_1 s^{2HK-2}\left( |u-1|^{2HK-2} - C_2 u^{2H-1}(u^{2H}+1)\right) \sim C_1 s^{2HK-2} |u-1|^{2HK-2}, u\to 1.
\end{gather*}
where $u=t/s$ and $C_1,C_2$ are some positive constants. Hence, $(B^+)$ holds on some interval $[t_0,1]$. As a result, we have Theorem~\ref{cor:main}  for $B^{H,K}$ with any $\rho>0$.

\subsubsection{Volterra integral transform of Wiener process}

Let $W=\{W(t), t\geq 0\}$ be a standard Wiener process. Consider  the processes of the form $X(t) = \int_0^t K(t,s) d W(s)$ with non-random kernels $K=\{K(t,s):[0,1]^2\rightarrow \mathbb{R}\}$ such that $K(t,\cdot)\in L^2{[0,t]}$ for any $t\in [0,1]$. Our goal is to establish the conditions on the kernel $K$ that supply $(A1)$, $(A2)$, and $(B^+)$, so that small ball property of Proposition \ref{smallballprop} is in place but only on the intervals separated from $0$. In what follows the constant $r\in[0,1/2)$ is fixed.

\begin{theorem}
\label{Th4.1}
Let  the kernel $K=\{K(t,s), t,s \in[0,1]\}$   satisfy   conditions
\begin{itemize}
\item[ $(B1)$] The kernel $K$ is non-negative on $[0,1]^2$  and for any $s\in [0,1]$   $K(\cdot,s)$ is non-decreasing in the first argument.
\item[$(B2)$]  There exist constants   $D_i>0, i=2,3$ and   $1/2<H_2<1$    such that   $$|K(t_2,s) - K(t_1,s)| \leq D_2 |t_2-t_1|^{H_2 }s^{-r},\;s,\; t_1,\;t_2 \in [0,1] $$
      and  $$\ K(t,s)\leq D_3(t-s)^{H_2-1/2}s^{-r},$$
\end{itemize}
      and at least one of the following conditions
\begin{itemize}
    \item[$(B3,a)$]  There exist constants  $D_1>0$ and $H_1\geq H_2 $   such that $$D_1|t_2-t_1|^{H_1}s^{-r}\leq|K(t_2,s) - K(t_1,s)|,\;s,\; t_1,\;t_2 \in [0,1];$$
     \item[$(B3,b)$]There exist constants  $D_1>0$ and $H_1\geq H_2 $   such that $$K(t,s)\geq D_1(t-s)^{H_1-1/2}s^{-r},\;s,\; t_1,\;t_2 \in [0,1].$$
\end{itemize} 
Then the Gaussian  process  $X(t) = \int_0^t K(t,s) dW(s)$, satisfies  conditions $(A1)$, $(A2)$, and $(B^+)$ on any subinterval $[1-\delta, 1]$ for $0<\delta<1$
  with powers $H_1, H_2$. % Moreover, if $H_2> 2H_1-1$, then the exponent  of $\varepsilon$ in small ball bound \eqref{equa} os negative (see Remark \ref{remark 2.2}).
\end{theorem}
\begin{proof} Note that  the process $X$ is correctly defined due to condition $K(t,\cdot)\in L^2{[0,t]}$ for any $t\in [0,1]$,  and it has a covariance function of the form
\begin{equation*}
\E X({t_1}) X({t_2}) =   \int_0^{t_1\wedge t_2 } K(t_1,z) K(t_2,z)  dz.
\end{equation*}
Further, for any $ 0 < t_1 < t_2 < 1$
\begin{equation*}
X_{t_2} - X_{t_1} = \int_0^{t_1}(K(t_2,z)- K(t_1,z))dW(z) + \int_{t_1}^{t_2} K(t_2,z) dW(z).
\end{equation*}
Therefore, for any    $s_1,t_1,s_2,t_2\in[0,1]$, $ s_1\leq t_1\leq s_2\leq t_2$   
\begin{equation}\begin{gathered}\label{eqaut11}
\E(X_{t_1}-X_{s_1})(X_{t_2}-X_{s_2}) =   \int_0^{s_1} (K(t_1,z)- K(s_1,z)) (K(t_2,z)- K(s_2,z)) dz\\+\int_{s_1}^{t_1} K(t_1,z)(K(t_2,z)- K(s_2,z)) dz.
\end{gathered}\end{equation}
If the kernel is non-negative and non-decreasing in the 1st argument, the right-hand side of \eqref{eqaut11} is non-negative so, condition $(B^+)$ holds.
Furthermore, for any $0\leq s<t\leq 1$
\begin{equation}\begin{gathered}\label{eqaut12}
\E (X(t)-X(s))^2 =   \int_0^{s} (K(t,z)- K(s,z))^2dz+\int_s^t K^2(t,z) dz.
\end{gathered}\end{equation}
Let $\delta\in [0,1]$ be fixed. Suppose that $s>1-\delta$ and estimate the right-hand side of \eqref{eqaut12} from above:
\begin{equation}\begin{gathered}\label{eqaut15}
\int_0^{s} (K(t,z)- K(s,z))^2dz+\int_s^t K^2(t,z) dz\leq
     D_2^2|t-s|^{2H_2}\int_0^sz^{-2r}dz\\+D_3^2 \int_s^tz^{-2r}(t-z)^{2H_2-1}dz \leq \frac{D_2^2}{1-2r}|t-s|^{2H_2}+D_3^2(1-\delta)^{-2r}|t-s|^{2H_2}=D_4|t-s|^{2H_2},
\end{gathered}\end{equation}
where $D_4=\frac{D_2^2}{1-2r}+D_3^2(1-\delta)^{-2r}$.
 The estimate under condition $(B3,a)$ is evident:
\begin{equation}\begin{gathered}\label{eqaut16}
\int_0^{s} (K(t,z)- K(s,z))^2dz+\int_s^t K^2(t,z) dz\geq  D_1^2|t-s|^{2H_1}\int_0^{1-\delta}z^{-2r}dz\\ \geq (1-\delta)^{1-2r}\frac{D_1^2}{1-2r}|t-s|^{2H_1}.
\end{gathered}\end{equation}
The estimate under condition $(B3,b)$ is also simple:
\begin{equation}\begin{gathered}\label{eqaut17}
\int_0^{s} (K(t,z)- K(s,z))^2dz+\int_s^t K^2(t,z) dz\geq D_1^2 \int_s^t (t-z)^{2H_1- 1}z^{-2r}dz\\ \geq (1-\delta)^{1-2r}\frac{D_1^2}{1-2r}|t-s|^{2H_1}.
\end{gathered}\end{equation}
Proof follows immediately from \eqref{eqaut15}--\eqref{eqaut17}.
\end{proof}
\begin{remark} Let $H\in(1/2,1)$ and consider the kernel
 $$K(t,s)=C_Hs^{1/2-H}\varphi(s)\int_s^tu^{H-1/2}(u-s)^{H-\frac32}du,$$ where $\varphi=\{\varphi(s), s\in[0,1]\}$
 is nonnegative measurable function satisfying assumption $0<k<\varphi(s)<K$, $C_H>0$ is some constant. Then the kernel is nonnegative and increasing in the first variable, so, condition $(B^+)$ holds and we can
 state with evidence that for any $0< z <s<t\leq 1$ $$|K(t,z)-K(s,z)|\leq C_H K s^{1/2-H}t^{H-1/2}(t-s)^{H-1/2}\leq C_H K s^{1/2-H}(t-s)^{H-1/2},$$
 so, we can put $r=H-1/2$ and $H_2=H-1/2$. Moreover, the  bound  from below has the form
 $$K(t,s)\geq C_Hk(t-s)^{H-1/2}\geq C_Hks^{1/2-H}(t-s)^{H-1/2},$$ and condition   $H_2> 2H_1-1$ is  satisfied since $H_1=H_2=H$. Therefore    process $X(t) = \int_0^t K(t,s) d W(s)$ has the properties supplying small ball bound. In the case when $\varphi(s)=1$ and the constant $C_H$ is defined correspondingly, process $\int_0^tK(t,s)dW(s)$ is a fractional Brownian motion. Now we can see that the small ball property of fBm is not a consequence of its stationary increments as one can deduce from previous results (\cite{li-shao}, for example), but of its property to be a quasi-helix. 
\end{remark}

\subsubsection{Volterra integral transform of fractional Brownian motion}
 Consider now fractional Brownian motion $B_H=\{B_H(t), t\in[0,1]\}$, the kernel $K=\{K(t,s):[0,1]^2\rightarrow \mathbb{R}\}$  and create Volterra fractional process  of the form $$X(t) = \int_0^t K(t,s) dB_H(s)$$ that exists under sufficient  condition $$\int_0^t\int_0^t |K(t,u) K(t,v)|  |u-v|^{2H-2}du dv<\infty, t\in[0,1].$$
 One of the simplest examples of Volterra fractional processes is  a fractional Ornstein-Uhlenbeck process $X=(X_t, t \in [0,T])$ satisfying the equation
\begin{equation}
X_t = X_0 + a \int_0^t X_s ds + B_t^H 
\label{OUeq}
\end{equation}
with   $a\in \mathbb{R}$. For $H>\frac{1}{2}$ the unique solution of equation \ref{OUeq} admits the representation
\begin{equation}
X_t = X_0 e^{at}+ e^{at}\int_0^t e^{-as}dB_s^H.
\end{equation}
Considering the particular case $X_0=0 $, we get the process $X_t = \int_0^t e^{a(t-s)} dB_s^H$ with kernel $K(t,s) = e^{a(t-s)}.$

\begin{theorem}
\label{Th4}
Let $H>\frac{1}{2}$ and let the kernel  $K=\{K(t,s), s,t\in [0,1]\}$   satisfy  $(B1)$ and the following conditions:
\begin{itemize}
\item[  $(B4)$] The kernel $K$ is bounded, i.e. there are constants $k$ and $K$ such that $0 < k \le K(t,s) \leq K$ for any $s,t \in [0,T]$;
 
\item[  $(B5)$] There exist constants  $C>0$ and $H_3 \geq 1/2$  for which  $|K(t_2,s) - K(t_1,s)| \leq C |t_2-t_1|^{H_3}$, $s$, $t_1$, $t_2 \in [0,1].$
\end{itemize}

Then the Gaussian  process $X$ with kernel $K$, $X_t = \int_0^t K(t,s) dB_s^H$, satisfies  conditions $(A1)$, $(A2)$ and $(B^+)$ with exponents $H_1 = H_3\wedge H$, $H_2 = H_3$.
\end{theorem}

\begin{proof}
Firstly, we note that  the process $X$ is correctly defined due to condition $(B1)$ and has a covariance function of the form ( see, e.g.,  \cite{Mish}, \cite{NVV})
\begin{equation}
\E X_{t_1} X_{t_2} = H(2H-1) \int_0^{t_1} \int_0^{t_2} K(t_1,s) K(t_2,v) |s-v|^{2H-2}ds dv.
\end{equation}
Further, for any $ 0 < t_1 < t_2 < 1$
\begin{equation}
X_{t_2} - X_{t_1} = \int_0^{t_1}(K(t_2,s)- K(t_1,s))dB_s^H + \int_{t_1}^{t_2} K(t_2,s) dB_s^H.
\end{equation}
Therefore, for   any  $s_1,t_1,s_2,t_2\in[0,1]$, $ s_1\leq t_1\leq s_2\leq t_2$   we have that
\begin{gather*} %\nonumber
 \E(X_{t_1}-X_{s_1})(X_{t_2}-X_{s_2})\\
 = \E \left( \int_0^{s_1}  \left(K(t_1,v) -K(s_1,v) \right)dB_v^H + \int_{s_1}^{t_1}  K(t_1,v) dB_v^H\right)
\\  \times\left( \int_0^{s_2}  \left(K(t_2,v) -K(s_2,v) \right)dB_v^H + \int_{s_2}^{t_2}  K(t_2,v) dB_v^H\right)\\\nonumber
= H (2H-1) \Big( \int_0^{s_1} \int_{0}^{s_2} \left( K(t_1,v) -K(s_1,v)\right)\left( K(t_2,z) -K(s_2,z)\right)|v-z|^{2H-2}dvdz \label{l1}\\ +\nonumber
\int_0^{s_1} \int_{s_2}^{t_2}  K(t_2,v)\left(K(t_1,z) -K(s_1,z)\right)|v-z|^{2H-2}dvdz
\\\nonumber
+ \int_{s_1}^{t_1} \int_0^{s_2} \left( K(t_2,v) -K(s_2,v)\right)K(t_1,z)|v-z|^{2H-2}dvdz\\\nonumber
+ \int_{s_1}^{t_2} \int_{s_2}^{t_2} K(t_1,v)  K(t_2,z) |v-z|^{2H-2}dvdz\Big)\geq 0,
 \end{gather*}
therefore, property $(B^+)$ holds.

Furthermore,
\begin{equation*}\begin{gathered}
\E |X_{t}- X_s|^2   
  =\E \Big|\int_0^{s }   ( K(t,v)- K(s ,v) ) dB_v^H  + \int_{s}^{t}  K(t,v) dB_v^H \Big|^2\\
 \leq  2 \E  \left|\int_0^{s }   ( K( t ,v)- K(s ,v) ) dB_v^H\right|^2  + 2\E\Big|\int_{s }^{t}  K(t,v) dB_v^H \Big|^2.
\end{gathered}\end{equation*}

According to \cite{MMV}, there exists a constant $C_H$ depending only on $H$ such that %{\color{red}(here [MMV]  ? what is it?)}
\begin{equation}
\begin{gathered}
\E  \left|\int_0^{s }   ( K(t,v)- K(s ,v) ) dB_v^H \right|^2 \\ \leq C_H \Big(\int_0^{s }  \left( K(t,v)- K(s,v)\right)^{\frac{1}{H}} dv\Big)^{2H} \leq C_H \, K^2 \,|t-s|^{2H_3} 
\end{gathered}\end{equation}

and
\begin{eqnarray*}
\E \left|\int_{s}^{t}  K(t,v) dB_v^H \right|^2 \leq C_H \left( \int_{s }^{t}  \left(K(t,v)\right)^{\frac{1}{H}} dv \right)^{2H} \leq C_H K^2 |t-s|^{2H},
\end{eqnarray*}
and condition $(A1) $ holds with $H_2=H_3\wedge H$. Finally, 
\begin{equation*}\begin{gathered}
\E |X_{t}- X_s|^2
  \geq C_H\int_s^{t} \int_s^{t}    K(t,v)K(t,z)|v-z|^{2H}dvdz\geq    k^2|t-s|^{2H}, 
\end{gathered}\end{equation*}
whence condition $(A2)$ follows. 

\end{proof}
\begin{remark}
Evidently, a fractional Ornstein-Uhlenbeck process with positive drift $a$ and zero initial value satisfies the assumptions  of the above theorem. 
  
But the representation result of Theorem~\ref{cor:main} is also valid for a fractional Ornstein--Uhlenbeck process with a negative drift. Indeed, by the fractional Girsanov theorem (see e.g.\ \cite{Mish}), its law is equivalent to that of the fractional Brownian motion driving it. Moreover, the fractional Brownian motion generates the same filtration as the fractional Brownian motion. Therefore, the representation theorem for fractional Brownian motion can be transfered to the fractional Ornstein--Uhlenbeck process via the Girsanov transform. Naturally, the same may be done for a wide class of processes, e.g.\ solutions of  stochastic differential equations with fractional Brownian motion. (See also the discussion in Remark~\ref{rem:equivalence} above.)
\end{remark}

\nocite{*}
\bibliographystyle{abbrv}
\bibliography{represent-2}

\begin{thebibliography}{10}

\bibitem{azmoo-vita}
E.~Azmoodeh and L.~Viitasaari.
\newblock A general approach to small deviation via concentration of measures.
\newblock 2014.
\newblock arXiv:math.PR/1407.3553.

\bibitem{bojdecki}
T.~Bojdecki, L.~G. Gorostiza, and A.~Talarczyk.
\newblock Sub-fractional {B}rownian motion and its relation to occupation
  times.
\newblock {\em Statist. Probab. Lett.}, 69(4):405--419, 2004.

\bibitem{dudley}
R.~M. Dudley.
\newblock {Wiener functionals as Ito integrals.}
\newblock {\em Ann. Probab.}, 5:140--141, 1977.

\bibitem{houdre}
C.~Houdr{\'e} and J.~Villa.
\newblock An example of infinite dimensional quasi-helix.
\newblock In {\em Stochastic models ({M}exico {C}ity, 2002)}, volume 336 of
  {\em Contemp. Math.}, pages 195--201. Amer. Math. Soc., Providence, RI, 2003.

\bibitem{kahane}
J.-P. Kahane.
\newblock H\'elices et quasi-h\'elices.
\newblock In {\em Mathematical analysis and applications, {P}art {B}}, volume~7
  of {\em Adv. in Math. Suppl. Stud.}, pages 417--433. Academic Press, New
  York-London, 1981.

\bibitem{kahanebook}
J.-P. Kahane.
\newblock {\em Some random series of functions}, volume~5 of {\em Cambridge
  Studies in Advanced Mathematics}.
\newblock Cambridge University Press, Cambridge, second edition, 1985.

\bibitem{li-shao}
W.~V. Li and Q.-M. Shao.
\newblock Gaussian processes: inequalities, small ball probabilities and
  applications.
\newblock In {\em Stochastic processes: theory and methods}, volume~19 of {\em
  Handbook of Statistics}, pages 533--597. North-Holland, Amsterdam, 2001.

\bibitem{MMV}
J.~M{\'e}min, Y.~Mishura, and E.~Valkeila.
\newblock Inequalities for the moments of {W}iener integrals with respect to a
  fractional {B}rownian motion.
\newblock {\em Statist. Probab. Lett.}, 51(2):197--206, 2001.

\bibitem{mish-shev-valk}
Y.~Mishura, G.~Shevchenko, and E.~Valkeila.
\newblock Random variables as pathwise integrals with respect to fractional
  {B}rownian motion.
\newblock {\em Stochastic Process. Appl.}, 123(6):2353--2369, 2013.

\bibitem{Mish}
Y.~S. Mishura.
\newblock {\em Stochastic calculus for fractional {B}rownian motion and related
  processes}, volume 1929 of {\em Lecture Notes in Mathematics}.
\newblock Springer-Verlag, Berlin, 2008.

\bibitem{NVV}
I.~Norros, E.~Valkeila, and J.~Virtamo.
\newblock An elementary approach to a {G}irsanov formula and other analytical
  results on fractional {B}rownian motions.
\newblock {\em Bernoulli}, 5(4):571--587, 1999.

\bibitem{samko}
S.~G. Samko, A.~A. Kilbas, and O.~I. Marichev.
\newblock {\em Fractional integrals and derivatives. Theory and applications}.
\newblock Gordon and Breach Science Publishers, Yverdon, 1993.

\bibitem{vita-shev}
G.~{Shevchenko} and L.~{Viitasaari}.
\newblock {Integral representation with adapted continuous integrand with
  respect to fractional Brownian motion.}
\newblock {\em {Stochastic Anal. Appl.}}, 32(6):934--943, 2014.

\bibitem{vita-shev2}
G.~Shevchenko and L.~Viitasaari.
\newblock Adapted integral representations of random variables.
\newblock {\em International Journal of Modern Physics: Conference Series}, 36,
  2015.

\bibitem{vita}
L.~Viitasaari.
\newblock Integral representation of random variables with respect to
  {G}aussian processes.
\newblock {\em Bernoulli}, 2015.
\newblock to appear, arXiv:math.PR/1307.7559.

\bibitem{zahle}
M.~Z{\"a}hle.
\newblock On the link between fractional and stochastic calculus.
\newblock In {\em Stochastic dynamics ({B}remen, 1997)}, pages 305--325.
  Springer, New York, 1999.

\end{thebibliography}

\end{document}